\documentclass[reqno]{amsart}
\usepackage{amsmath}
\usepackage{amsthm}
\usepackage{amsfonts}
\usepackage{amssymb}
\usepackage{graphicx}
\usepackage{mathrsfs}
\usepackage{cite}
\usepackage{textcomp}

\title[STOLARSKY'S INVARIANCE PRINCIPLE FOR PROJECTIVE SPACES]{STOLARSKY'S INVARIANCE PRINCIPLE FOR PROJECTIVE SPACES}

\author[M.M. SKRIGANOV]{M.M. SKRIGANOV}

\address{St. Petersburg Department of the Steklov Mathematical Institute 
of the Russian Academy of Sciences, 
27, Fontanka, St.Petersburg 191023, Russia}

\email{mmskrig@gmail.com}



\numberwithin{equation}{section}

\newtheorem{theorem}{Theorem}[section]
\newtheorem{lemma}{Lemma}[section]
\newtheorem{proposition}{Proposition}[section]
\newtheorem{corollary}{Corollary}[section]
\theoremstyle{remark}

\theoremstyle{remark}
\newtheorem{definition}{Definition}[section]

\def\dd{\mathrm{d}}

\def\Cc{\mathbb{C}}

\def\Ff{\mathbb{F}}

\def\Hh{\mathbb{H}}

\def\Oo{\mathbb{O}}

\def\Rr{\mathbb{R}}

\def\Zz{\mathbb{Z}}

\def\BBB{\mathcal{B}}
\def\CCC{\mathcal{C}}
\def\DDD{\mathcal{D}}
\def\EEE{\mathcal{E}}

\def\III{\mathcal{I}}

\def\MMM{\mathcal{M}}

\DeclareMathOperator{\diam}{diam}

\renewcommand{\le}{\leqslant}
\renewcommand{\ge}{\geqslant}

\numberwithin{equation}{section}

\DeclareMathOperator{\Spin}{\mathrm{Spin}}
\DeclareMathOperator{\RE}{\mathrm{Re}}
\DeclareMathOperator{\Tr}{\mathrm{Tr}}

\def\M{\mathcal M}

\def\F{\mathbb F}
\def\bR{\mathbb R}
\def\la{\lambda}

\def\D{\mathcal D}

\def\E{\mathcal E}

\def\HH{\mathcal H}

\def\bO{\mathbb O}
\def\bR{\mathbb R}

\def\bZ{\mathbb Z}

\numberwithin{equation}{section}
\theoremstyle{plain}

\newcommand{\bp}{\begin{proof}}
\newcommand{\ep}{\end{proof}}
\newcommand{\bl}{\begin{lemma}}
\newcommand{\el}{\end{lemma}}
\newcommand{\bt}{\begin{theorem}}
\newcommand{\et}{\end{theorem}}
\newcommand{\bd}{\begin{definition}}
\newcommand{\ed}{\end{definition}}
\newcommand{\ba}{\begin{arrow}}
\newcommand{\ea}{\end{arrow}}

\begin{document}


%
%
%
%




\begin{abstract}

We show that Stolarsky's invariance
principle, known for point distributions on the Euclidean spheres, 
can be extended to the real, complex, and quaternionic  
projective spaces and the octonionic projective plane. 
(The paper will be published in \emph{Journal of Complexity},  2020.)
\end{abstract}

\thanks
{This work is supported by the Program of the Presidium of the Russian Academy of Sciences 
``New Methods of Mathematical Modeling in the Study of Nonlinear Dynamical Systems" under Grant PRAS 08-04.}


\maketitle

\thispagestyle{empty}



\section{Introuction and main results}\label{sec1}

Let $S^d=\{x\in \Rr^{d+1}:\Vert x\Vert=1\}$ be the standard $d$-dimensional unit sphere in $\Rr^{d+1}$
with the geodesic (great circle) metric $\theta$ and the Lebesgue measure
$\mu$ normalized by $\mu(S^d)=1$. 
We write $\CCC(y,t)=\{x\in S^d: (x,y) > t\}$ for the spherical cap
of height $t\in [-1,1]$ centered at 
$y\in S^d$. Here we write $(\cdot,\cdot)$ and $\Vert \cdot \Vert$ for 
the inner product and the Euclidean norm in $\Rr^{d+1}$.

For an $N$-point subset $\DDD_N\subset S^d$,
the spherical cap quadratic discrepancy is defined by
\begin{equation}
\lambda^{cap}[\DDD_N]=
\int_{-1}^1\int_{S^d}\left(\,\# \{|\CCC(y,t)\cap \DDD_N\}-N\mu(\CCC(y,t))\,\right)^2\dd\mu(y)\,
\dd t.
\label{eq1.1*}
\end{equation}

We introduce the sum of pairwise Euclidean distances 
between points of $\DDD_N$
\begin{equation}
\tau [\DDD_N]=\frac 12 \sum\nolimits_{x_1, x_2\in \DDD_N} \Vert x_1-x_2 \Vert
= \sum\nolimits_{x_1, x_2\in \DDD_N} \sin\frac 12 \theta (x_1,x_2),
\label{eq1.2*}
\end{equation}
and write $\langle \tau \rangle$ for the average value of the Euclidean
distance on $S^d$,       
\begin{equation}
\langle \tau \rangle =\frac 12 \iint\nolimits_{S^d\times S^d} \Vert y_1-y_2\Vert \,
d\mu (y_1) \, \dd\mu (y_2).
\label{eq1.12*}
\end{equation}

The study of the quantities \eqref{eq1.1*} and \eqref{eq1.2*} falls within 
the subjects of discrepancy theory and geometry of distances,
see  \cite{2, 6} and references therein.  
It turns out that the quantities \eqref{eq1.1*} and \eqref{eq1.2*} are 
not independent and are intimately 
related by the following remarkable identity
\begin{equation}
\gamma(S^d)\lambda^{cap}[\DDD_N]+\tau [\DDD_N]=\langle \tau \rangle N^2,
\label{eq1.3**}
\end{equation}
for an arbitrary $N$-point subset $\DDD_N\subset S^d$. Here $\gamma(S^d)$
is a positive constant independent of $\DDD_N$,
\begin{equation}
\gamma(S^d)=\frac{d\,\sqrt{\pi}\,\,\Gamma(d/2)}{2\,\Gamma((d+1)/2)}\,\,
\thicksim\,\,\sqrt{\pi d/2}  \, .
\label{eq1.3***}
\end{equation}

The identity \eqref{eq1.3**} was established  by Stolarsky \cite{33},
and known in the literature as Stolarsky's invariance principle. 
Its original proof has been essentially simplified 
by Brauchart and Dick \cite{11}. Further simplifications were given in  
the paper \cite{8} by Bilyk, Dai and Matzke.
Particularly, the explicit formula \eqref{eq1.3***} has been given in \cite{8, 11}.
In our notation $\gamma (S^d)=(2C_d)^{-1}$, where $C_d$ is the constant
in \cite[Theorem~2.2]{8} and \cite[Eq. (6)]{11}.

In the present paper we consider the relations of this type in
a more general setting. 
Let $\M$ be a compact metric measure space with a fixed metric $\theta$ 
and a finite Borel measure $\mu$, normalized, for convenience, by 
\begin{equation}
\diam (\MMM,\theta)=\pi, \quad \mu (\MMM)=1,
\label{eq1.1}
\end{equation}
where 
$\diam (\EEE,\rho)=\sup \{ \rho(x_1,x_2): x_1,x_2\in \EEE\}$
denotes the diameter of a subset $\EEE\subseteq \MMM$  with respect to 
a metric $\rho$. 

We write $\BBB(y,r)=\{x\in\MMM:\theta (x,y)<r\}$ for the ball of radius $r\in \III$
centered at  $y\in \MMM$ and of volume $v(y,r)=\mu (\BBB(y,r))$. 
Here $\III = \{r=\theta(x_1,x_2): x_1,x_2\in \MMM\}$ denotes the set of all possible radii. 
If the space $\MMM$ is connected, we have $\III = [0,\pi]$.

We consider distance-invariant metric spaces. Recall that a metric space $\MMM$ is called distance-invariant, if the volume  of any
ball $v(r)=v(y,r)$ is independent of $y\in \MMM$, see \cite[p.~504]{24}.
The typical examples of distance-invariant spaces are homogeneous spaces $\MMM=G/K$, 
where $G$ is a compact group, $K\subset G$ is a closed subgroup, and a metric
$\theta$ and a measure $\mu$ on $\MMM$ are $G$-invariant. 

For an $N$-point subset $\DDD_N\subset \MMM$, the ball quadratic discrepancy is 
defined by
\begin{equation}
\lambda[\xi,\DDD_N]=
\int_{\III}\int_{\MMM}\left(\,\# \{\BBB(y,r)\cap \DDD_N\}-Nv(r))\,\right)^2\,
\dd\mu(y)\, \dd \xi (r),
\label{eq1.3*}
\end{equation}
where $\xi$ is a measure on the set of radii $\III$.

Notice that for $S^d$ spherical caps and balls are 
related by  $\CCC(y,t)=\BBB(y,r)$, $t=\cos r$, and  the 
discrepancies \eqref{eq1.1*} and \eqref{eq1.3*} are related by
$\lambda^{cap}[\DDD_N]=\lambda[\xi^{\natural},\DDD_N]$, where  
$\dd\xi^{\natural}(r)=\sin r\,\dd r,\, r\in \III =[0,\pi]$.

The ball quadratic discrepancy \eqref{eq1.3*} can be written in the form
\begin{equation}
\lambda [\xi, \DDD_N] =
\sum\nolimits_{x_1,x_2\in \DDD_N} \la (\xi, x_1,x_2)
\label{eq1.7}
\end{equation}
with the kernel 
\begin{equation}
\lambda(\xi,x_1,x_2)=\int_\III\int_\MMM
\Lambda  (\BBB(y,r),x_1)\,\Lambda (\BBB(y,r),x_2)
\, \dd\mu (y)\,\dd\xi (r)\, ,
\label{eq1.6}
\end{equation}
where 
\begin{equation}
\Lambda  (\BBB(y,r),x)=\chi(\BBB(y,r),x)- v(r),
\label{eq1.6**}
\end{equation}
and $\chi(\E,\cdot)$ denotes the characteristic function of a subset $\E\subseteq\M$.

The symmetry of the metric $\theta$ implies the following relation   
\begin{equation}
\chi (\BBB(y,r),x)=\chi (\BBB(x,r),y)=\chi_0 (r-\theta (x,y)),
\label{eq1.16*}
\end{equation}
where $\chi_0(\cdot)$ is the characteristic function of the half-axis 
$(0,\infty)$.
Substituting \eqref{eq1.6**} into \eqref{eq1.6} and using \eqref{eq1.16*}, we obtain
\begin{equation}
\lambda(\xi,x_1,x_2)=\int_\III
\Big (\mu  (\BBB(x_1,r)\cap \BBB(x_2,r)) -v(r)^2 \Big)
\, \dd\xi (r)
\label{eq1.6*}
\end{equation}

For an arbitrary metric $\rho$ on $\MMM$ we introduce 
the sum of pairwise distances  
\begin{equation}
\rho [\DDD_N] =\sum\nolimits_{x_1,x_2\in \DDD_N} \rho (x_1,x_2).
\label{eq1.10}
\end{equation}
and the average value 
\begin{equation}
\langle \rho \rangle =\int\nolimits_{\MMM\times\MMM} \rho (y_1,y_2) \,
\dd\mu (y_1) \, \dd\mu (y_2).
\label{eq1.12}
\end{equation}
We introduce the following symmetric difference metrics on the space $\M$
\begin{align}
\theta^{\Delta} (\xi, y_1,y_2) & =\frac 12\int_\III
\mu (\BBB(y_1,r)\Delta \BBB(y_2,r))
\, \dd\xi(r) \notag
\\
& =\frac 12\int_\III\int_{\MMM}\chi(\BBB(y_1,r)\Delta \BBB(y_2,r),y)
\,\dd\mu(y)\,\dd\xi(r),
\label{eq1.13}
\end{align}
where 
$$\BBB(y_1,r)\Delta \BBB(y_2,r)=\BBB(y_1,r)\cup \BBB(y_2,r) \setminus 
\BBB(y_1,r)\cap \BBB(y_2,r)$$
is the symmetric difference of the balls  $\BBB(y_1,r)$ and $\BBB(y_2,r)$.
We have
\begin{equation}
\chi(\BBB(y_1,r)\Delta \BBB(y_2,r),y) = 
\vert\chi (\BBB(y_1,r),y)-\chi(\BBB(y_2,r),y)\vert.
\label{eq1.15}
\end{equation}
Therefore,
\begin{equation}
\theta^{\Delta} (\xi, y_1,y_2)=
\frac 12\int_\III\int_{\MMM}\vert\chi (\BBB(y_1,r),y)-\chi(\BBB(y_2,r),y)\vert
\,\dd\mu(y)\,\dd\xi(r).
\label{eq1.15*}
\end{equation}
On the other hand, we have 
\begin{align}
& \chi (\BBB(y_1,r)\Delta \BBB(y_2,r))   \notag
\\
& =\chi(\BBB(y_1,r),y) +\chi 
(\BBB(y_2,r),y) -2\chi (\BBB(y_1,r),y) \chi (\BBB(y_2,r),y).
\label{eq1.17*}
\end{align}
Substituting \eqref{eq1.17*} into  \eqref{eq1.13} and using \eqref{eq1.16*},
we obtain 
\begin{equation}
\theta^{\Delta}(\xi,y_1,y_2) 
=\int_{\III} \Big(v(r)-\mu (\BBB(y_1,r)\cap \BBB(y_2,r))\Big)\, \dd \xi (r),
\label{eq1.14*}
\end{equation}
and 
\begin{align}
\langle \theta^{\Delta}(\xi) \rangle  
= \int_{\III}\Big(v(r)-v(r)^2\Big)\, \dd\xi (r).
\label{eq1.19}
\end{align}

In line with the definition \eqref{eq1.10}, we put 
\begin{equation*}
\theta^{\Delta} [\xi,\DDD_N] =\sum\nolimits_{x_1,x_2\in \DDD_N} 
\theta^{\Delta} (\xi,x_1,x_2).
\end{equation*}

Comparing the relations \eqref{eq1.6*}, \eqref{eq1.14*}, and \eqref{eq1.19}, 
we arrive at the following. 

\begin{proposition}\label{prop1.1}
Let a compact metric measure space $\M$ with a metric $\theta$
and a measure $\mu$  be distance-invariant. Then we have
\begin{align}
\lambda (\xi,y_1,y_2)+\theta^{\Delta}(\xi ,y_1,y_2) 
= \langle 
\theta^{\Delta} (\xi)\rangle .
\label{eq1.30}
\end{align}
In particular, we have the following $L_1$-invariance principle
\begin{equation}
\la [\,\xi,\DDD_N\,]+\theta^{\Delta}[\,\xi , \DDD_N\,]  = \langle 
\theta^{\Delta} (\xi)\rangle \, N^2
\label{eq1.31}
\end{equation}
for an arbitrary $N$-point subset $\DDD_N\subset \MMM$. 

The identities \eqref{eq1.30} and \eqref{eq1.31} hold 
with any measure $\xi$ on the set of radii $\III$ such that 
the integrals \eqref{eq1.6*}, \eqref{eq1.14*} and \eqref{eq1.19} converge 
\textup{(}for example, with any finite measure $\xi$\textup{)}.  
\end{proposition}
Other versions and applications of this result can be found in \cite{30}.

Recall that a metric space $\MMM$ with a metric $\rho$ is called isometrically
$L_q$-embeddable ($q=1\,\, \mbox {or}\,\, 2$), if there exists 
a map $\varphi:\MMM\ni x\to \varphi(x)\in L_q$, such that
$\rho(x_1,x_2)=\Vert\varphi(x_1)-\varphi(x_2)\Vert_{L_q}$ for all $x_1$, 
$x_2\in \M$. Notice that the  $L_2$-embeddability 
is stronger and
implies the $L_1$-embeddability, see~\cite[Sec.~6.3]{17}. 

Since the space $\MMM$ is isometrically $L_1$-embeddable with respect to
the symmetric difference metrics $\theta^{\Delta} (\xi)$, see \eqref{eq1.15*},
the identity \eqref{eq1.31} is called the $L_1$-invariance principle. 
At the same time, Stolarsky's invariance principle should be called 
the $L_2$-invariance principle, because it involves the Euclidean metric.

In the present paper we shall prove the $L_2$-invariance principles for compact
Riemannian symmetric manifolds of rank one. All these manifolds are
completely classified, see, for example,~\cite[Chap.3]{7} and~\cite[Sec.~8.12]{36}.
They are homogeneous spaces $\MMM=G/K$, where  $G$ and $K\subset G$ are
compact Lie groups. The complete list of all compact
Riemannian symmetric manifolds of rank one is the following: 

(i) The $d$-dimensional Euclidean spheres 
$S^d=SO(d+1)/SO(d)\times 
\{1\}$, $d\ge 2$, and $S^1=O(2)/O(1) \times \{ 1\}$. 

(ii) The real projective spaces $\Rr P^n=O(n+1)/O(n)\times O(1)$.

(iii) The complex projective spaces $\Cc P^n=U(n+1)/U(n)\times U(1)$.

(iv) The quaternionic projective spaces $\Hh P^n=Sp(n+1)/Sp(n)\times Sp(1)$,

(v) The octonionic projective plane $\Oo P^2=F_4/\Spin (9)$.

Here we use the standard notation from the theory of Lie groups; in particular,  
$F_4$ is one of the exceptional Lie groups in Cartan's classification. 

The indicated projective spaces $ \Ff P^n $ as compact Riemannian manifolds have dimensions $d$, 
\begin{equation}
d=\dim_{\Rr} \Ff P^n=nd_0, \quad d_0=\dim_{\Rr}\Ff,
\label{eq2.1} 
\end{equation}
where $d_0=1,2,4,8$ for $\Ff=\Rr$, $\Cc$, $\Hh$, $\Oo$, correspondingly.

For the spheres $S^d$ we put $d_0=d$ by definition. Projective spaces 
of dimension  $d_0$ ($n=1$) are homeomorphic to the spheres $S^{d_0}$:
$\Rr P^1 \approx S^1, \Cc P^1 \approx S^2,  \Hh P^1 \approx S^4, 
\Oo P^1 \approx S^8$.
We can conveniently agree that $d>d_0$ ($n\ge 2)$ for projective spaces,
while the equality $d=d_0$ holds only for spheres. Under this convention,
the dimensions $d=nd_0$ and $d_0$ define uniquely (up to homeomorphisms) 
the corresponding homogeneous space which we denote by $Q=Q(d,d_0)$.

We consider $Q(d,d_0)$ as a metric measure space with the metric $\theta$
and measure $\mu$ proportional to the invariant Riemannian distance and measure
on $Q(d,d_0)$. The coefficients of proportionality are defined to satisfy \eqref{eq1.1}.
In what follows we always assume that $n=2$ if $\Ff=\Oo$, since  
projective spaces $\Oo P^n$ do not exist for $n>2$. 

The spaces $Q(d,d_0)$ have a very rich geometrical structure and can be also
characterized as compact connected two-point homogeneous spaces. This means
that for any two pairs of points $x_1$, $x_2$ and 
$y_1$, $y_2$ in $Q(d,d_0)$ with $\theta(x_1,x_2)=\theta(y_1,y_2)$ there exists an isometry $g\in G$, such that $y_1=gx_1$, $y_2=gx_2$.
In more detail the geometry 
of spaces  $\Ff P^n$ will be outlined in Section 2.

Any space $Q(d,d_0)$ is distance-invariant and the volume of balls is given by 
\begin{equation}
v(r)=B(d/2,d_0/2)^{-1}
\int^r_0(\sin\frac{1}{2}u)^{d-1}(\cos \frac{1}{2}u)^{d_0-1}\,\dd u, 
\quad r\in [0,\pi],
\label{eq2.2}
\end{equation}
where $B(\cdot,\cdot)$ is the beta function, see \eqref{eq0.1}. 
Equivalent forms of  \eqref{eq2.2}
can be found in the literature, see \cite[pp.~177--178]{19}, \cite[pp.~165--168]{22},  \cite[pp.~508--510]{24}.

The chordal metric on the spaces $Q(d,d_0)$ can be defined by
\begin{equation}
\tau(x_1,x_2)=\sin \frac{1}{2}\theta(x_1,x_2), \quad x_1,x_2\in Q(d,d_0).
\label{eq2.4}
\end{equation}
Notice that the expression \eqref{eq2.4} defines a metric because the function 
$\varphi(\theta)=\sin \theta/2$, $0\le \theta\le \pi$, 
is concave, increasing, and  $\varphi(0)=0$, that implies the triangle 
inequality.
For the sphere $S^d$ we have
$\cos \theta(x_1,x_2)=(x_1,x_2),\, x_1,x_2\in S^d$ and
\begin{equation}
\tau(x_1,x_2)=\sin \frac12\theta(x_1,x_2)=\frac{1}{2}\,\Vert x_1-x_2\Vert\, .
\label{eq2.6*}
\end{equation}
Each projective space $\Ff P^n$ can be canonically embedded into the unit sphere
\begin{equation}
\Pi:Q(d,d_0)\ni x\to \Pi(x)\in S^{m-1}\subset \Rr^m, \quad 
m=\frac{1}{2}(n+1)(d+2),
\label{eq2.6}
\end{equation}
such that
\begin{equation}
\tau(x_1,x_2)=\frac{1}{\sqrt{2}}\|\Pi(x_1)-\Pi(x_2)\|, \quad x_1,x_2\in 
\Ff P^n,
\label{eq2.7}
\end{equation}
where $\|\cdot \|$ is the Euclidean norm in $\Rr^{m}$.
Hence, the metric $\tau(x_1,x_2)$ is proportional to the Euclidean length of a segment
joining the corresponding points $\Pi(x_1)$ and $\Pi(x_2)$ on the unit sphere 
and normalized by $\diam (Q(d,d_0),\tau)=1$.
The embedding \eqref{eq2.6} will be described explicitly in Section~\ref{sec2}.

The chordal metric $\tau$ on the complex projective space $\Cc P^n$ 
is known as the Fubini--Study metric. 
The chordal metric on projective spaces has been discussed in the papers \cite{13, 14}
in connection with special point configurations in such spaces.
The chordal metric has been also defined for Grassmannian manifolds 
in \cite{15}.

Now we are in position to state our main result. 

\begin{theorem}\label{thm1.1}
For any space $Q=Q(d,d_0)$ the chordal metric
\eqref{eq2.4} and the symmetric difference metric \eqref{eq1.13} are related by
\begin{equation}
\tau(x_1,x_2)=\gamma(Q)\,\, \theta^{\Delta}(\xi^{\natural} ,x_1,x_2), \ 
x_1,x_2\in Q,
\label{eq2.8}
\end{equation}
where $\dd\xi^{\natural}(r)=\sin r\,\dd r$, $r\in [0,\pi]$, and 
\begin{equation}
\gamma(Q)=\frac{\langle \tau\rangle}{\langle 
\theta^{\Delta}(\xi^{\natural})\rangle}=\frac{\diam 
(Q,\tau)}{\diam(Q,\,\theta^{\Delta}(\xi^{\natural}))}.
\label{eq2.9}
\end{equation}
\end{theorem} 
              
The proof of Theorem~\ref{thm1.1} is given in Section~\ref{sec3}.
It is clear that the equalities \eqref{eq2.9} follow immediately from \eqref{eq2.8}. 
It suffices to calculate the average values \eqref{eq1.12} of both metrics 
in \eqref{eq2.8} to obtain 
the first equality in \eqref{eq2.9}. Similarly, writing \eqref{eq2.8} for any pair of 
antipodal points $x_1$, $x_2$, $\theta(x_1,x_2) = \pi$, we obtain the second 
equality in \eqref{eq2.9}. Recall that points $x_1, x_2$ are antipodal for a metric
$\rho$ if $\rho(x_1,x_2) = \diam(Q,\rho)$. 
If points $x_1, x_2$ are antipodal for
the metric $\theta$, then, in view of \eqref{eq2.4} and \eqref{eq2.8}, they are also
antipodal for the metrics $\tau$ and $\theta^{\Delta}(\eta^{\natural})$.

Comparing Theorem~1.1 and Proposition~1.1, we arrive at the following.

\begin{corollary}\label{cor2.1}
For any space $Q=Q(d,d_0)$ we have the $L_2$-invariance principle
\begin{equation}
\gamma(Q)\,\lambda [\xi^{\natural},\D_N]+\tau [\D_N]=\langle \tau\rangle 
N^2
\label{eq2.10}
\end{equation}
for an arbitrary $N$-point subset $\D_N\subset Q$ .
\end{corollary}

The identity \eqref{eq2.10} can be thought of as an extension of
Stolarsky's invariance principle to all compact
Riemannian symmetric manifolds of rank one.

Now we wish to calculate 
the constants $\langle \tau\rangle$ and $\gamma(Q)$ in 
the invariance principle \eqref{eq2.10}. 
Using \eqref{eq2.2}, \eqref{eq2.4}, and the formula \eqref{eq0.1} for 
the beta function, we immediately obtain
$$\langle \tau\rangle 
= B(d/2, d_0 /2)^{-1}\, B((d+1)/2, d_0 /2)\, .$$  
The explicit calculation of the constant $\gamma (\F P^n)$ is more differentiated.
In principle, for this purpose, one can use the first equality 
in \eqref{eq2.9} 
and the formula \eqref{eq1.19}. 
In the case of
 $\Ff P^n,\, \Ff \neq \Rr $, the integrals \eqref{eq2.2} have rather simple
explicit expressions, see \cite[p.~341]{19aa}. 
To calculate the constant $\gamma(\Rr P^n)$ one can use the integral
representation  \eqref{eq6.7***} given below. 
However, such calculations turn out to be rather cumbersome.
In the present paper, we shall use a more general approach relying on
the theory of spherical functions on the spaces $Q(d,d_0)$ to obtain the following.
\begin{theorem}\label{thm1.2}
For any space $Q(d,d_0)$, we have
\begin{align}
\gamma(Q(d,d_0))
=\frac{\sqrt{\pi}}{4}\,(d+d_0)\,
\frac{\Gamma(d_0/2)}{\Gamma((d_0+1)/2)}
=\frac{d+d_0}{2d_0}\,\gamma(S^{d_0})\, ,
\label{eq1.33*}
\end{align}
where $\gamma(S^{d_0})$ is defined by \eqref{eq1.3***}.
\end{theorem}

Theorem~1.2 is proved in section~4.
For the sphere $S^d=Q(d,d)$, the relation \eqref{eq1.33*} coincides with 
the formula \eqref{eq1.3***}.
For projective spaces, from \eqref{eq1.33*} we obtain the following.
\begin{corollary}\label{cor2.2}
For projective spaces $\F P^n=Q(nd_0,d_0),\,d_0=\dim_{\Rr}\F$, we have 
\begin{equation}
\gamma(\F P^n)=\frac{n+1}{2}\, \gamma (S^{d_0})\, ,
\label{eq1.33}
\end{equation}
and therefore, 
\begin{equation}
\left.
\begin{aligned}
& \gamma (\Rr P^n)  = \frac{n+1}{2}\, \gamma (S^{1}) = \frac{\pi}{4} \, (n+1)\, ,
\\
& \gamma (\Cc P^n)=  \frac{n+1}{2}\, \gamma (S^{2}) = n+1\, ,
\\
& \gamma (\Hh P^n) =  \frac{n+1}{2}\, \gamma (S^{4}) = \frac{4}{3}\, (n+1)\, ,
\\
& \gamma (\Oo P^2) = \,\,\frac{3}{2}\, \gamma (S^{8}) = \frac{192}{35}\, .
\end{aligned}
\label{eq1.34}
\right\}
\end{equation}
\end{corollary}

Notice that the comparison of
the formulas \eqref{eq1.3***} with \eqref{eq1.33} and \eqref{eq1.34}
shows that for  spheres and projective spaces the behavior 
of the constants $\gamma(Q(d,d_0))$ differs essentially in large
dimensions.

In conclusion of this section, we briefly describe applications of
the invariance principle \eqref{eq2.10} to uniform point distributions
in the spaces $Q=Q(d,d_0)$. 
With the help of the invariance principle \eqref{eq2.10**}, the results of distance geometry can be transformed to the corresponding results of discrepancy theory and vice versa.

Let us consider the following extremal quadratic discrepancies 
and sums of pair-wise chordal distances 
\begin{equation}
\lambda_N (Q) =\inf\nolimits_{\DDD_N} \lambda [\xi^{\natural}, \DDD_N]\, , \quad
\tau_N(Q)=\sup\nolimits_{\DDD_N}\tau [\DDD_N]\, ,
\label{eq1.9}
\end{equation}
where the infimum and supremum are taken over all $N$-point subsets 
$\DDD_N\subset Q$. These quantities can be thought of as 
geometrical characteristics of the spaces $Q$.

From \eqref{eq2.10}, we obtain the identity
\begin{equation}
\gamma(Q)\,\lambda_N(Q)+\tau_N(Q)=\langle \tau\rangle N^2 .
\label{eq2.10**}
\end{equation}
 
First of all, the identity \eqref{eq2.10**} implies the non-trivial inequality  
$\tau_N(Q)\le\langle \tau \rangle N^2$. 
This inequality can be essentially improved. We have the upper bound
\begin{equation}
\langle \tau \rangle N^2 -\tau_N(Q) \lesssim  N^{1-\frac{1}{d}}\, ,
\label{eq1.3***b}
\end{equation}
and by  \eqref{eq2.10**}, we obtain 
\begin{equation}
\lambda_N(Q)\lesssim N^{1-\frac{1}{d}} \, .
\label{eq1.3***a}
\end{equation}
On the other hand, we have the lower bound 
\begin{equation}
\lambda_N (Q)\gtrsim  N^{1-\frac{1}{d}} \, ,
\label{eq1.3****b}
\end{equation}
and by \eqref{eq2.10**}, we obtain
\begin{equation}
\langle \tau \rangle N^2 -\tau [\DDD_N] \gtrsim N^{1-\frac{1}{d}}\, .
\label{eq1.3****a}
\end{equation}
Combining the above bounds \eqref{eq1.3***b}---\eqref{eq1.3****a}, we find
the sharp orders of the extremal quantities \eqref{eq1.9}
\begin{equation}
\left.
\begin{aligned}
& \langle \tau \rangle N^2 -\tau_N(Q) \simeq  N^{1-\frac{1}{d}}\, ,
\\
& \lambda_N (Q)\simeq  N^{1-\frac{1}{d}} \, .
\end{aligned}
\label{eq1.34a}
\right\}
\end{equation}

For the spheres $S^d$, the upper bounds \eqref{eq1.3***b} and \eqref{eq1.3***a}
have been established by Alexander \cite{1} and Stolarsky \cite{33}, while
the lower bounds \eqref{eq1.3****b} and \eqref{eq1.3****a}
have been established by Beck \cite{5}, see also \cite{2, 6}.
For all spaces $Q(d,d_0)$ the bounds 
\eqref{eq1.3***b} -- \eqref{eq1.3***a} 
were proved in \cite{31}. Notice that the upper bounds of the type 
\eqref{eq1.3***b} and \eqref{eq1.3***a} can be established for very general
compact metric measure spaces, see \cite{5*, 30, 32}. At the same time,
the lower bounds  \eqref{eq1.3****b} and \eqref{eq1.3****a}
are much more specific. Their proof relies on harmonic analysis 
on the homogeneous spaces $Q(d,d_0)$, see \cite{31}, and can not be 
extended to general compact metric measure spaces.

The paper is organized as follows. In Section~2 we define and discuss
the chordal metrics on the projective spaces
$\Ff P^n$, $\F=\Rr$, $\Cc,$ $\Hh$, $n\ge 2$, and the octonionic projective 
plane  $\Oo P^2$ in terms of special models for such spaces. 
For the reader's convenience, we describe these 
models in close detail and give the necessary references.
In Section~3 we prove Theorem~1.1 relying on the results of Section~2 and 
a special representation for symmetric difference metrics (Lemma~3.1).
For completeness, in Section~3 we give a short proof of 
Stolarsky's invariance principle for the spheres $S^d$. 
A very short proof of Theorem~1.1 can be given in the special case
of real projective spaces.
This simple proof, specific for $\Rr P^n$, is also given in Section~3.
In Section~4 we calculate the constants $\gamma(Q(d,d_0))$ and 
prove Theorem~1.2. The proof relies on the zonal spherical
function expansions for the chordal and symmetric difference metrics
(Lemmas 4.1 and 4.2).
In conclusion of Section~4, we briefly discuss
explicit formulas for integrals with Jacobi polynomials
which follow from our calculations (Remark~4.2).


\section{Models of projective spaces and chordal metrics }\label{sec2}

Recall the general facts on the division algebras $\Rr, \Cc, \Hh, \Oo$ 
over the field of real numbers. 
We have the natural inclusions $\Rr\subset \Cc\subset \Hh\subset \Oo$,
where the octonions $\Oo$ are a nonassociative and noncommutative algebra
of  dimension 8  with a basis $\{1, e_1, e_2, e_3, e_4, e_5, e_6, e_7\}$ 
(the multiplication table for these elements can be found, for example, 
in \cite[p.~150]{3} and \cite[p.~90]{7}),
the quaternions $\Hh$ are an associative but noncommutative subalgebra of
dimension 4 spanned by $\{1,e_1, e_2, e_3\}$, finally, $\Cc$ and $\Rr$ are
associative and commutative subalgebras of dimensions 2 and 1 spanned by 
$\{1, e_1\}$ and $\{1\}$. 
From the multiplication table one can easily see that for any two indexes
$1 \le i,j \le 7,i\ne j, $ there exists an index $1 \le k \le 7$, such that
\begin{equation}
e_ie_j=-e_je_i=e_k, \quad i < j, \quad e^2_i=-1.
\label{eq5.1}
\end{equation}
Let $a=\alpha_0+\sum\nolimits_{i=1}^{7}\alpha_{i} e_i, \, 
\alpha_i\in \Rr$, $0\le i\le 7$, be a typical octonion. We write  $\RE a=\alpha_0$  
for the real part,  
$\bar a=\alpha_0-\sum\nolimits^{7}_{i=1} \alpha_ie_i$ for the conjugation, 
$\vert a\vert=\big (\alpha^2_0+\sum\nolimits^{7}_{i-1}\alpha^2_i\big)^{1/2}$ for the norm. 
Using \eqref{eq5.1}, one can easily check that 
\begin{equation}
\RE ab=\RE ba,\quad \overline{ab}=\overline{ba}, \quad 
\vert a\vert^2=a\bar a=\bar aa, \quad \vert ab\vert=\vert a\vert\,\vert b\vert. 
\label{eq5.2}
\end{equation}
Notice that by a theorem of Artin a subalgebra in $\Oo$ generated 
by any two octonions is associative and isomorphic to one of the algebras 
$\Hh$, $\Cc$, or $\Rr$, see \cite{3}.
 
The usual model of projective spaces over
the associative algebras $\Ff=\Rr$, $\Cc$, $\Hh$ is the following, 
see \cite{3, 7, 20, 36}. 
Let $\F^{n+1}$ be a linear
space of vectors $\mathbf a=(a_0,\dots,a_n)$, $a_i\in \Ff$, $0\le i\le n$
with the right multiplication by scalars $a\in\Ff$, the Hermitian
inner product 
\begin{equation}
(\mathbf a,\mathbf b) =\sum\nolimits^{n}_{i=0} \bar a_i b_i ,\quad
\mathbf a,\mathbf b \in \F^{n+1},
\label{eq5.3}
\end{equation}
and the norm $\vert \mathbf a\vert$,
\begin{equation}
\vert\mathbf a\vert^2=(\mathbf a, \mathbf a) =
\sum\nolimits^{n}_{i=0} \vert a_i\vert^2. 
\label{eq5.4}
\end{equation}

In view of associativity of the algebras $\Ff=\Rr$, $\Cc,\Hh$,  
a projective space $\Ff P^n$ can be defined as a set of 
one-dimensional (over  $\Ff$) subspaces in $\Ff^{n+1}$: 
\begin{equation}
\Ff P^n=\{ p(\mathbf a)=\mathbf a \Ff : \mathbf a \in \Ff^{n+1}, \, 
|\mathbf a|=1\} .
\label{eq5.5}
\end{equation}
The metric $\theta$ on $\Ff P^n$ is defined by 
\begin{equation}
\cos \frac 12 \theta (\mathbf a,\mathbf b)\!=\!|(\mathbf a,\mathbf b)|, \quad
\mathbf a,\mathbf b\in \Ff^{n+1}, \quad  |\mathbf a|\!=\!|\mathbf b|\!=\!1, \,\,  
0\le \theta (\mathbf a,\mathbf b)\le \pi .
\label{eq5.6}
\end{equation}
In other words,  $\frac 12 \theta (\mathbf a,\mathbf b)$ is the angle between  
the subspaces $p(\mathbf a)$ and $p(\mathbf b)$.
The transitive group of isometries   $U(n+1, \Ff)$ for the metric $\theta$ consists of 
nondegenerate linear transformations of the space  $\Ff^{n+1}$, preserving  
the inner product  \eqref{eq5.3}, and the stabilizer of a point is isomorphic to
the subgroup $U(n,\Ff)\times U(1,\Ff)$. Hence, 
\begin{equation}
\F P^n=U(n+1, \Ff)/ U(n,\Ff) \times U(1,\Ff).
\label{eq5.7}
\end{equation}
The groups $U(n+1,\Ff)$ are indicated in Section~1. 

There is another model where a projective space $\Ff P^n$, $\Ff=\Rr, \Cc, 
\Hh$, is identified
with the set of orthogonal projectors onto the subspaces 
$p(\mathbf a)\subset\Ff^{n+1}$. 
This model admits a generalization to the octonionic projective plane 
$\Oo P^2$ 
and in its terms the chordal metric can be naturally defined for all 
such projective spaces.

Let $\HH (\F^{n+1})$ denote the set of all Hermitian ${(n+1)}\times (n+1)$ matrices 
with the entries in $\Ff$, $\F=\Rr,\Cc,\Hh,\Oo$, 
\begin{equation}
\HH (\Ff^{n+1})= \{ A=((a_{ij})) : a_{ij}=\overline{a}_{ji}, \;
a_{ij}\in \Ff, \, 0\le i,j\le n \}\, ,  
\label{eq5.8}
\end{equation}
where $n=2$ if $\Ff=\Oo$.
It is clear that $\HH(\Ff^{n+1})$ is a linear space over $\Rr$ of dimension 
\begin{equation}
m=\dim_{\Rr}\HH (\Ff^{n+1})=\frac 12 (n+1) (d+2), \quad d = n d_0.
\label{eq5.9}
\end{equation}

The space $\HH (\F^{n+1})$ is equipped  with the symmetric real-valued inner product 
\begin{equation}
\langle A,B\rangle =\frac 12 \Tr (AB+BA)= \RE \Tr AB=
\RE \sum\nolimits^{n}_{i,j=0} a_{ij}\overline{b_{ij}}
\label{eq5.10}
\end{equation}
and the Hilbert -- Schmidt norm 
\begin{equation}
\Vert A\Vert =(\Tr A^2)^{1/2} =  \left(  \sum\nolimits^{n}_{i,j=0}
|a_{ij}|^2\right)^{1/2},
\label{eq5.11}
\end{equation}
where $\Tr A=\sum\nolimits^{n}_{i=0}a_{ii}$  denotes the trace of 
a matrix $A$. 
For the distance $\Vert A-B\Vert$ between two matrices $A,B\in \HH (\F^{n+1})$, we have 
\begin{equation}
\Vert A-B\Vert^2 =\Vert A\Vert^2 +\Vert B\Vert^2 -2 \langle A,B\rangle\, .
\label{eq5.12}
\end{equation}
Thus, $\HH (\Ff^{n+1})$ can be thought of as the $m$-dimensional Euclidean space. 

If $\Ff\ne \Oo$, the orthogonal projector  
$\Pi_{\mathbf a}\in \HH (\Ff^{n+1})$ onto
$p(\mathbf a)=\mathbf a \Ff$, $\mathbf a=(a_0,\dots,a_n)\in \Ff^{n+1}$,
$\vert\mathbf a\vert=1$, can be given by $\Pi_{\mathbf a}=\mathbf a (\mathbf a,\cdot)$
or as the $(n+1)\times (n+1)$ matrix  with entries
$(\Pi_{\mathbf a})_{i,j}=a_i\bar a_j$, $0\le i,j\le n$. 
Therefore, the projective space 
\eqref{eq5.5} can be written as follows
\begin{equation}
\Ff P^n=\{ \Pi\in \HH (\F^{n+1}): \Pi^2=\Pi, \, \,  \Tr \Pi =1\} .
\label{eq5.13}
\end{equation}
One can easily check that for such projectors the inner products \eqref{eq5.3}
and \eqref{eq5.10} are related by 
$\langle \Pi_{\mathbf a},\Pi_{\mathbf b}\rangle =\vert (\mathbf a,\mathbf b)\vert^2$,
see \cite[Eq. (2.1)]{14},
and the group of isometries $U(n+1,\F)$ acts on such projectors by
the formula $g(\Pi)=g\Pi g^{-1}$, $g\in U(n+1,\Ff)$. 

For the octonionic projective plane $\bO P^2$ a similar model is also true.
A detailed discussion of this model can be found
in \cite{3, 7, 20}, including an explanation why octonionic projective
spaces $\bO P^n$ do not exist if $n>2$.
In this model one puts by definition 
\begin{equation}
\Oo P^2=\{ \Pi \in \HH (\Oo^3): \Pi^2=\Pi, \, \,  \Tr \Pi=1\}.
\label{eq5.14}
\end{equation}
The formulas \eqref{eq5.13} and \eqref{eq5.14} are quite similar. 
One can check that each matrix in \eqref{eq5.14} can be written as 
$\Pi_{\mathbf a}\in \Oo P^2$ for a triple $\mathbf a=(a_0,a_1,a_2)\in 
\Oo^3$, where 
$(\Pi_{\mathbf a})_{i,j}=a_i\bar a_j$, 
$0\le i,j\le 2$, 
$\vert\mathbf a\vert^2=\vert a_0\vert^2+\vert a_1\vert^2+\vert a_2\vert^2=1$, 
and additionally  
$(a_0a_1)a_2=a_0(a_1a_2)$, see  \cite[Lemma 14.90]{20}. The additional condition 
means that the subalgebra in $\Oo$ generated by the elements $a_0,a_1,a_2$
is associative. Using this fact, one can show that $\Oo P^2$  is
a 16-dimensional compact connected Riemannian manifold, see \cite{3, 7, 20}. 

The group of nondegenerate linear transformations $g$ of the space 
$\HH(\Oo^3)$ 
preserving the squares  $g(A^2)=g(A)^2$, $A\in \HH(\Oo^3)$, is isomorphic to
the 52-dimensional exceptional Lie group  $F_4$. This group also preserves 
the trace, inner product \eqref{eq5.10} and norm \eqref{eq5.11} of matrices 
$A\in \HH (\Oo^3)$. 
The group $F_4$ is transitive on  $\Oo P^2$, and the stabilizer of a point is 
isomorphic to the spinor group $\Spin (9)$, 
see  \cite[Lemma 14.96 and Theorem 14.99]{20}.  
Hence, $\Oo P^2=F_4/\Spin (9)$ is a homogeneous space, and one can prove that
$\Oo P^2$ is a two-point homogeneous space. 


Now we wish to describe the structure of geodesics in projective spaces. 
Such a description can be easily given in terms
of the models \eqref{eq5.13} and \eqref{eq5.14}. It is known, 
see \cite{7, 21, 36}, 
that all geodesics on a two-point homogeneous space $Q(d,d_0)$ are closed
and homeomorphic to the unit circle.  The group of isometries is transitive
on the set of geodesics and the stabilizer of a point is transitive
on the set of geodesics passing through this point. Therefore,
all geodesics have the same length $2\pi$ (under the normalization \eqref{eq1.1}).

The inclusions $\Rr\subset \Cc\subset \Hh\subset \Oo$ induce the following inclusions 
of the corresponding projective spaces 
\begin{equation}
\Ff_1P^{n_1} \subseteq \Ff P^n, \quad \Ff_1\subseteq \Ff, \quad n_1\le n.
\label{eq5.15}  
\end{equation}
Furthermore, the subspace $\Ff_1P^{n_1}$ is a geodesic submanifold 
in $\Ff P^n$,
see \cite[Sec.~3.24]{7}. Particularly, the real projective line $\Rr P^1$ is 
homeomorphic to the unit circle $S^1$ and can be embedded as a geodesic into
all projective spaces $\Ff P^n$, 
\begin{equation}
S^1 \approx \Rr P^1 \subset \Ff P^n,
\label{eq5.16}
\end{equation}
see \cite[Proposition 3.32]{7}. In \eqref{eq5.16} $n=2$ if $\Ff=\Oo$. These facts can also be immediately derived from a general description of geodesic submanifolds in Riemannian symmetric spaces, see 
\cite[Chap. VII, Corollary 10.5]{21}.

Using the models \eqref{eq5.13} and \eqref{eq5.14}, we can write the real 
projective line $\Rr P^1$ as the following set of $2\times 2$ matrices: 
\begin{equation}
\Rr P^1=\{\zeta (u), u\in \Rr /\pi \bZ\} ,
\label{eq5.17}
\end{equation}
where
\begin{equation*}
\zeta (u)  \!=\! 
\begin{pmatrix}
\cos^2 u & \sin u \cos u  \\  \sin u \cos u & \sin^2 u
\end{pmatrix} \! 
\\
=\!
\begin{pmatrix}
\cos u & -\sin u \\ \sin u  & \cos u 
\end{pmatrix}    
   \begin{pmatrix}
1 & 0 \\ 0 & 0 
\end{pmatrix}
\begin{pmatrix}
\cos u & \sin u \\ \sin u & \cos u
\end{pmatrix}.
\end{equation*}   
For each $u\in \Rr$ the matrix $\zeta (u)$ is an orthogonal projector onto   
the one-dimensional subspace $x\Rr$, $x=(\cos u, \sin u)\in S^1$. 
The embedding $\Rr P^1$ into $\Ff P^n$ can be written as the following set of 
$(n+1)\times (n+1)$ matrices 
\begin{equation}
Z= \{ Z(u), u\in \Rr /\pi \Zz\} \subset \Ff P^n,
\label{eq5.18}
\end{equation}
where
\begin{equation*}
Z(u)= 
\begin{pmatrix}
\zeta (u) & 0_{n-1,2}  \\  0_{2, n-1}  & 0_{n-1,n-1}
\end{pmatrix},
\end{equation*}
and $0_{k,l}$ denotes the zero matrix of size $k\times l$. 
The set of matrices \eqref{eq5.18}  is a geodesic in $\Ff P^n$. 
All other geodesics are of the form $g(Z)$, where $g\in G$ is an isometry
of the space $\F P^n$. 
The parameter $u$ in \eqref{eq5.18} and the geodesic distance $\theta$ on the space 
$\F P^n$ are related by 
\begin{equation}
\theta (Z(u), Z(0)) =2 \vert u\vert, \quad - \pi/2 <u\le  \pi/2,        
\label{eq5.19}
\end{equation}
and for all $u\in \Rr$ this formula can be extended by periodicity.
In particular, we have
\begin{equation}
\theta (Z(v), Z(-v)) = 4v,  \quad 0\le v \le \pi/4. 
\label{eq5.20}
\end{equation}
The relation \eqref{eq5.20} will be needed in the next section.

Now, we define the chordal distance on projective spaces.
The formulas \eqref{eq5.13}, \eqref{eq5.14} and \eqref{eq5.11} imply 
\begin{equation}
\Vert \Pi\Vert^2 = \Tr \Pi^2 = \Tr \Pi =1.
\label{eq5.21}
\end{equation}
for any $\Pi\in \Ff P^n$. Therefore, the projective spaces $\Ff P^n$,
defined by \eqref{eq5.13} and \eqref{eq5.14}, are submanifolds in the unit sphere 
\begin{equation}
\Ff P^n\subset S^{m-1} =\{ A\in \HH (\Ff^{n+1}) :\Vert A\Vert =1\} \subset
\HH (\Ff^{n+1})\approx \bR^m.
\label{eq5.22}
\end{equation}
In fact, the formula \eqref{eq5.22} defines an embedding of $\Ff P^n$ into 
the $(m-2)$-dimensional sphere,
the intersection of the sphere $S^{m-1}$ and the hyperplane 
in $\HH(\Ff^{n+1})$ given by $\Tr A=1$.

The chordal distance $\tau (\Pi_1,\Pi_2)$  between 
$\Pi_1,\Pi_2\in \Ff P^n$ is defined as the Euclidean distance \eqref{eq5.12}:
\begin{equation}
\tau (\Pi_1,\Pi_2)= \frac{1}{\sqrt{2}} \Vert \Pi_1-\Pi_2\Vert =
(1-\langle \Pi_1,\Pi_2\rangle )^{1/2}.
\label{eq5.23}
\end{equation}
The coefficient $1/\sqrt{2}$ is chosen to satisfy $\diam (\Ff P^n, \tau)=1$. 

It follows from \eqref{eq5.23} that  $\tau (g(\Pi_1)$, 
$g(\Pi_2))=\tau (\Pi_1,\Pi_2)$  
for all isometries $g\in G$ of the space $\Ff P^n$.
Since $\Ff P^n$ is a two-point homogeneous space, for any 
$\Pi_1,\Pi_2\in \Ff P^n$ with $\theta (\Pi_1,\Pi_2)=2u$, 
$0\le u\le  \pi /2$, there exists an isometry $g\in G$, such that 
$g(\Pi_1)=Z(u)$, $g(\Pi_2)=Z(0)$.  From   \eqref{eq5.23}, \eqref{eq5.18} and \eqref{eq5.17}, 
we obtain 
$
\tau (Z(u),Z(0)) =\sin u =\sin \frac12\theta (\Pi (u),\Pi(0)).
$
Therefore, 
$\tau (\Pi_1,\Pi_2)= \sin \frac12 \theta (\Pi_1,\Pi_2)$,
as it was defined before in \eqref{eq2.4}.
Notice also that pairs of antipodal points $\Pi_+,\Pi_-\in \Ff P^n$ 
(with  $\theta (\Pi_+,\Pi_-)=\pi$ and $\tau (\Pi_+,\Pi_-)=1$) 
can be characterized by the orthogonality condition 
$\langle \Pi_+,\Pi_-\rangle =0$, see \eqref{eq5.23}.

\section{Proof of Theorem 1.1}\label{sec3}

The proof of Theorem~\ref{thm1.1} relies on the following special representation
of the symmetric difference metric \eqref{eq1.13}. 
\begin{lemma}\label{lem6.1} For a distance--invariant space $\MMM$, we have  
\begin{equation}
\theta^{\Delta} (\xi,y_1,y_2)=\frac12 \int_{\M} | \sigma
(\theta(y_1,y)) - \sigma (\theta (y_2,y)) | \, \dd\mu (y)
\label{eq6.1}
\end{equation}
with the non-increasing function 
$\sigma(r)= \xi ([r,\pi ]) = \int^{\pi}_{r}  \, \dd \xi (u)$.

In particular, for a homogeneous space $Q=Q(d,d_0)$
and the measure $\dd\xi^{\natural}(r)=\sin r\,\dd r,\, r\in [0,\pi]$, we have   
\begin{equation}
\theta^{\Delta} (\xi^{\natural}, y_1,y_2)= \int_Q 
|\tau (y_1,y)^2-\tau (y_2,y)^2| \, \dd\mu (y),
\label{eq6.3}
\end{equation}
where $\tau$ is the chordal metric \eqref{eq2.4} on $Q(d,d_0)$. 
\end{lemma}
Lemma~\ref{lem6.1} was proved earlier in \cite[Lemma 2.1]{30}. Here this 
result is given in a form adapted to the chordal metric.
For the spheres $S^d$, a formula similar to \eqref{eq6.3} was 
given earlier in \cite[Lemma 2.4]{8}.

\begin{proof}[Proof of Lemma~3.1] For brevity, we write $\theta (y_1,y)=\theta_1$ and 
$\theta(y_2,y)=\theta_2$. Using \eqref{eq1.13}, \eqref{eq1.16*} and \eqref{eq1.17*}, 
we obtain 
\begin{align}
&\theta^{\Delta} (\xi ,y_1,y_2)
\notag
\\
& = \frac12 \int_{\M} \left( \int^{\pi}_{0} (\chi_0  
(r-\theta_1) +\chi_0 (r-\theta_2) -2\chi_0 (r-\theta_1)\chi_0
(r-\theta_2))  \dd\xi(r)\,  \right) \, \dd\mu (y)
\notag
\\
& = \frac 12 \int_{\M} (\sigma (\theta_1) +\sigma (\theta_2) 
-2\sigma (\max \{ \theta_1,\theta_2\} )) \, \dd\mu (y).
\label{eq6.4}
\end{align}
Since $\sigma$ is a non-increasing function, we have 
\begin{equation}
2\sigma (\max\{ \theta_1,\theta_2\})\! =\! 2\min \{ \sigma (\theta_1),\sigma
(\theta_2)\} 
\!= \!\sigma (\theta_1)\!+\!\sigma (\theta_2)\! -\!|\sigma (\theta_1)\!-\!\sigma 
(\theta_2)| .
\label{eq6.5}
\end{equation}
Substituting \eqref{eq6.5} into \eqref{eq6.4}, we obtain \eqref{eq6.1}. 

If  $\dd\,\xi^{\natural}(r)=\sin r\,\dd r$, then the corresponding 
$\sigma (r)= 2-2(\sin r/2)^2$. Substituting this expression 
into \eqref{eq6.1} and using the definition \eqref{eq2.4}, we obtain \eqref{eq6.3}. 
\end{proof}

\begin{proof}[Proof of Theorem~\ref{thm1.1} for spheres]  By \eqref{eq2.6*}, 
we have
\begin{align}
& \tau (y_1,y)^2-\tau (y_2,y)^2 =\frac14 (\|y_1-y\|^2-\|y_2-y\|^2) 
\notag
\\
& = \frac12 (y_2-y_1,y) =\tau (y_1,y_2) (x,y), \quad y_1,y_2\in S^d,
\label{eq6.6}
\end{align}
where $x=\Vert y_2-y_1\Vert^{-1}(y_2-y_1)\in S^d$. Substituting 
\eqref{eq6.6} into \eqref{eq6.3}, we obtain 
\begin{equation}
\theta^{\Delta}(\xi^{\natural},y_1,y_2)= \tau (y_1,y_2)
\int_{S^d} |(x,y)| \, \dd\mu (y).
\label{eq6.7}
\end{equation}
The integral in \eqref{eq6.7} is independent of $x\in S^d$.
This proves the equality \eqref{eq2.8} for $S^d$ with the constant 
\begin{equation}
\gamma (S^d) =\left( \int\nolimits_{S^d}  |(x,y)| \,\dd\mu (y)\right)^{-1}.
\label{eq6.7*} 
\end{equation}
This completes the proof.
\end{proof}
The integral \eqref{eq6.7*} can be easily calculated
to obtain \eqref{eq1.3***}, see \cite{8, 11}.

\begin{proof}[Proof of Theorem~1.1 for real projective spaces]
The space $\Rr P^d=S^d/S^0,\, S^0=\{1,-1\}$ consists of the one dimensional 
subspaces $y=\mathbf a\Rr\subset \Rr ^{d+1}$, where $\mathbf a\in S^d$, 
and $\mathbf a$ and $-\mathbf a$ are identified. 
Consider two subspaces
$y_1=\mathbf a_1\Rr, \,\, y_2=\mathbf a_2\Rr\subset \Rr P^d$, 
and let $(\mathbf a_1,\mathbf a_2)\ge 0$. 
The definitions \eqref{eq2.4} and \eqref{eq5.6} imply
$$
\tau(y_1,y_2)^2=(\sin \frac12 \theta(y_1,y_2))^2 = 
1-(\mathbf a_1,\mathbf a_2)^2 .
$$ 
Putting $y=\mathbf a\Rr$, we obtain
\begin{align}
& \tau (y_2,y)^2-\tau (y_1,y)^2 = 
(\mathbf a_1,\mathbf a)^2 -(\mathbf a_2,\mathbf a)^2
\notag
\\
&=[(\mathbf a_1,\mathbf a) +(\mathbf a_2,\mathbf a)]\,
[(\mathbf a_1,\mathbf a) -(\mathbf a_2,\mathbf a)]=
(\mathbf a_1+\mathbf a_2,\mathbf a)\,(\mathbf a_1-\mathbf a_2,\mathbf a)
\notag
\\
& = 2\sin\frac12\theta(y_1,y_2) \,
(\mathbf a_+,\mathbf a)\,(\mathbf a_-,\mathbf a)=
2\tau (y_1,y_2)\, (\mathbf a_+,\mathbf a)\,(\mathbf a_-,\mathbf a),
\label{eq6.6*}
\end{align}
where $\mathbf a_1+\mathbf a_2=2\cos\frac{1}{4}\theta(y_1,y_2)\,\mathbf a_+$ 
and
$\mathbf a_1-\mathbf a_2=2\sin\frac{1}{4}\theta(y_1,y_2)\,\mathbf a_-$. 
Here $\mathbf a_+,\mathbf a_- \in S^d$ and 
$(\mathbf a_+,\mathbf a_-)=0$. 
The corresponding
mutually orthogonal subspaces 
$y_+=\mathbf a_+\Rr,\,\, y_-=\mathbf a_-\Rr$ are antipodal points
in $\Rr P^d:\,\, \theta(y_+,y_-)=\pi$. 

Substituting \eqref{eq6.6*} into \eqref{eq6.3}, we obtain
\begin{equation}
\theta^{\Delta}(\xi^{\natural},y_1,y_2)=2 \tau (y_1,y_2)
\int_{S^d/S^0} \,\vert (\mathbf a_+,\mathbf a)\,(\mathbf a_-,\mathbf a)\vert \, 
\dd\mu (\mathbf a).
\label{eq6.7**}
\end{equation}
Since $\Rr P^d$ is a two-point homogeneous space,
the integral in \eqref{eq6.7**} is independent of the choice of mutually
orthogonal unit vectors $\mathbf a_+, \mathbf a_-\in S^d$.
This proves the equality \eqref{eq2.8} for $\Rr P^d$ with the constant 
\begin{align}
\gamma (\Rr P^d) &=\left( 2 \int\nolimits_{S^d/S^0}  \,
\vert (\mathbf a_+,\mathbf a)\,(\mathbf a_-,\mathbf a)\vert \,
\dd\mu (\mathbf a)\right)^{-1}
\notag
\\
&=\left(2\, \omega_d ^{-1} \int\nolimits_{S^d}  \,
\vert (\mathbf a_+,\mathbf a)\,(\mathbf a_-,\mathbf a)\vert \,
\dd \mathbf a\right)^{-1} ,
\label{eq6.7***} 
\end{align}
where $\dd\mathbf a$ denotes the standard surface measure on $S^d$ and $\omega_d$
is the full surface measure of $S^d$. 
This completes the proof.
\end{proof}

The integral in \eqref{eq6.7***} can be easily calculated  
to obtain \eqref{eq1.34} for $\Rr P^d$.

\begin{proof}[Proof of Theorem~1.1 for general projective spaces]
We write $\Pi_1,\Pi_2,\Pi$ for points in the models of projective spaces \eqref{eq5.13}
and \eqref{eq5.14}.
With this notation, the relation \eqref{eq6.3} takes the form  
\begin{equation}
\theta^{\Delta}(\xi^{\natural},\Pi_1,\Pi_2)= \int_{\Ff P^n}
|\tau (\Pi_1,\Pi)^2 -\tau (\Pi_2,\Pi)^2 | \, \dd\mu (\Pi)\, .
\label{eq6.8}
\end{equation} 

Since $\Ff P^n$ is a two-point homogeneous space, for any
$\Pi_1,\Pi_2 \in \Ff P^n$ 
with $\theta (\Pi_1,\Pi_2)=4v$, $0\le v\le \pi/4$, there exists 
an isometry $g\in G$, 
such that $g(\Pi_1)=Z(v)$, $g(\Pi_2)=Z(-v)$, see \eqref{eq5.20}. Therefore,
\begin{equation}
\int_{\Ff P^n} |\tau (\Pi_1,\Pi)^2-\tau (\Pi_2,\Pi)^2 | \, \dd\mu 
(\Pi)   
\\
=\int_{\Ff P^n}|\tau (Z(v),\Pi)^2-\tau (Z(-v),\Pi)^2|\,\dd\mu (\Pi).
\label{eq6.9}
\end{equation}
From the definition \eqref{eq5.23}, we obtain 
\begin{align}
 \tau (Z(v), \Pi)^2\!-\!\tau (Z(-v), \Pi)^2
&\! =\!  \frac12\, (\,\| Z(v)-\Pi\|^2\! -\! \| Z(-v)\!-\!\Pi \|^2\,) \notag
\\
& = \langle\, Z(v)-Z(-v), \Pi \, \rangle .
\label{eq6.10}
\end{align}
The formulas \eqref{eq5.17} and \eqref{eq5.18} imply 
$$
Z(v)-Z(-v)=
\begin{pmatrix}
\zeta(v) -\zeta(-v) & 0_{n-1,2} \\
0_{2,n-1}  &  0_{n-1,n-1}
\end{pmatrix}
$$
and 
$$
\zeta(v)-\zeta(-v)=
\begin{pmatrix}
0  &  \sin 2v \\ \sin 2v & 0
\end{pmatrix}
= (\sin 2v )\,(\zeta_+ -\zeta_-), 
$$
where 
$$
\zeta_+= \frac12 
\begin{pmatrix}
1 & 1 \\ 1 & 1
\end{pmatrix} , \quad \zeta_-=\frac12 
\begin{pmatrix}
1  &-1 \\ -1 & 1
\end{pmatrix} .
$$
Therefore, 
\begin{equation}
Z(v)-Z(-v) =(\sin 2v)\, (Z_+-Z_-),
\label{eq6.11}
\end{equation}
where
$$
Z_{\pm} = 
\begin{pmatrix}
\zeta_{\pm} & 0_{n-1,2} \\  0_{2,n-1}  & 0_{n-1,n-1} 
\end{pmatrix} . 
$$
We have $Z^*_{\pm}=Z_{\pm}$, $Z^2_{\pm}=Z_{\pm}$,
$\Tr Z_{\pm}=1$, Therefore, $Z_{\pm}\in \F P^n$, and 
$\langle Z_+,Z_-\rangle =0$. This means that $Z_+$ and  $Z_-$ are antipodal 
points in $\Ff P^n$. 
Using \eqref{eq2.4}, we can write
$$
\tau (\,\Pi_1,\Pi_2\,)= \tau (Z(v),  Z(-v)) =\sin 2v,
$$
and the equality \eqref{eq6.11} takes the form
\begin{equation}
Z(v)-Z(-v)=\tau (\Pi_1,\Pi_2)\, (Z_+-Z_-).
\label{eq6.12}
\end{equation}
Substituting \eqref{eq6.12} into \eqref{eq6.10}, we find that   
\begin{equation}
\tau (Z(v),\Pi)^2 -\tau (Z(-v),\Pi)^2
= \tau (\Pi_1,\Pi_2)\,\langle\, Z_+-Z_-,\Pi \,\rangle  .
\label{eq6.13}
\end{equation}
Substituting \eqref{eq6.13} into \eqref{eq6.9} and using \eqref{eq6.8}, we obtain 
\begin{equation}
\theta^{\Delta}(\xi^{\natural}, \Pi_1,\Pi_2)= \tau
(\Pi_1,\Pi_2) \,\theta^{\Delta} (\xi^{\natural}, Z_+, Z_-),
\label{eq6.14}
\end{equation}
where 
\begin{equation}
\theta^{\Delta}(\xi^{\natural}, Z_+,Z_-)=
\int_{\Ff P^n}  |\langle\, Z_+-Z_-,\Pi\,\rangle | \, \dd\mu (\Pi).
\label{eq6.15}
\end{equation}
The integral \eqref{eq6.15} is independent of $\Pi_1$ and $\Pi_2$. 
This proves the equality \eqref{eq2.8} for  $\Ff P^n$ with the constant
\begin{equation}
\gamma (\Ff P^n) = \left( \,\int\nolimits_{\Ff P^n} \vert \langle\, Z_+-Z_-,
\Pi \,\rangle \vert \, \dd\mu (\Pi) \right)^{-1}. 
\label{eq6.16}
\end{equation}
Notice that in this formula any pair of antipodal points in $\Ff P^n$ can be taken 
instead of $Z_+,Z_-$.
The proof of Theorem~\ref{thm1.1} is complete. 
\end{proof}


\section{Proof of Theorem 1.2}\label{sec4}


The zonal spherical functions $\phi_l$ for the spaces $Q=Q(d,d_0)$ are 
eigenfunctions of the radial part of the Laplace--Beltrami operator on $Q$ 
and can be given explicitly, see 
\cite[p.~178]{19}, 
\cite[Chap.~V, Theorem~4.5]{22},
\cite[pp.~514--512, 543--544]{24}, \cite[Theorem~11.4.21]{37}.
We have
\begin{equation}
\phi_l(x_1,x_2) =
\phi_l(Q,x_1,x_2) = \frac{P^{(\alpha,\beta)}_l (\cos\theta(x_1,x_2))}
{P^{(\alpha,\beta)}_l (1)}, \quad l\ge 0,\,\, x_1, x_2 \in Q,
\label{eq4.1}
\end{equation}
where $P^{(\alpha,\beta)}_l (t),\, t\in [-1,1],$ are  Jacobi polynomials 
of degree $l$ with parameters
\begin{equation}
\alpha = d/2 - 1,\quad \beta  = d_0/2  -1,\quad \alpha ,\,\beta \ge -1/2.
\label{eq8.24}
\end{equation}

A detailed consideration of Jacobi polynomials can be found in \cite{1*, 2*, 34}.
They can be given by Rodrigues' formula 
\begin{equation}
P^{(\alpha,\beta)}_l(t) =
\frac{(-1)^l}{2^ll!} (1-t)^{-\alpha} (1+t)^{-\beta}
\frac{d^l}{dt^l} 
\left\{ (1-t)^{l+\alpha} (1+t)^{l+\beta} \right\}.
\label{eq9.11}
\end{equation}
We also have the bound
\begin{equation}
\vert P^{(\alpha,\beta)}_l(t)\vert 
\le P^{(\alpha,\beta)}_l(1) =
\begin{pmatrix} \alpha +l \\ l  \end{pmatrix}  
= O( l^{\alpha}),
\quad t\in [-1,1], 
\label{eq8.23}
\end{equation}
see \cite[Theorem~7.32.1]{34}.

Jacobi polynomials form a complete orthogonal system in the $L_2$-space on 
the segment $[-1,1]$ with
the weight $(1-t)^{\alpha}(1+t)^{\beta}$. We have the orthogonality relations,
\begin{align}
\int^1_{-1}
P^{(\alpha,\beta)}_l(t) P^{(\alpha,\beta)}_{l'} (t)
(1-t)^{\alpha}(1+t)^{\beta} \, \dd t = \delta_{ll'}\,2^{\alpha+\beta+1}\,M_l^{-1}\,, 
\label{eq8.25}
\end{align}                                               
where $\delta_{ll'}$ is Kronecker's symbol and
\begin{equation}
M_l=(2l+\alpha+\beta+1) 
\frac{\Gamma (l+1)\Gamma (l+\alpha+\beta+1)}{\Gamma(l+\alpha+1)
\Gamma (l+\beta+1)}\,= \, O(l),
\label{eq8.26}
\end{equation}
see \cite[Eq.~(4.3.3)]{34}. 
Notice that the asymptotic approximations in \eqref{eq8.23} and \eqref{eq8.26}
follow easily from Stirling's formula.

Using the orthogonality relations \eqref{eq8.25},
we obtain the following formal expansion 
\begin{equation}
f(t)=\sum\nolimits_{l\ge 0}\,2^{-\alpha-\beta-1}\,M_l
\, c_l\,P^{(\alpha,\beta)}_l(t),
\label{eq0.17**}
\end{equation}
for an integrable function $f(t),\,t\in [-1,1]$, where
\begin{equation}
c_l=\int^1_{-1}f(t)\,
(1-t)^{\alpha}(1+t)^{\beta}\, P^{(\alpha,\beta)}_l(t)\, \dd t .
\label{eq0.18**}
\end{equation}
If all derivatives $f^{(l)}(t)$ exist, and 
$f^{(l)}(t)\,(1-t)^{\alpha +l}(1+t)^{\beta +l}$ vanish at $t=\pm 1$ for all $l$, 
then substituting 
Rodrigues' formula \eqref{eq9.11} into \eqref{eq0.18**} and integrating
$l$ times by part, we obtain 
\begin{equation}
c_l=\frac{1}{2^ll!}\int^1_{-1}f^{(l)}(t)\,(1-t)^{\alpha +l}(1+t)^{\beta +l}\,\dd t.
\label{eq0.19**}
\end{equation}

In what follows, we always assume that the parameters $\alpha,\beta$ and 
the dimensions $d$, $d_0$ are related by \eqref{eq8.24}. We shall also use 
the well-known formulas for the beta function 
\begin{align}
B(a,b)= \frac{\Gamma (a)\Gamma (b)}{\Gamma (a+b)}&=
\int^{\pi}_0(\sin\frac{1}{2}u)^{2a-1}(\cos \frac{1}{2}u)^{2b-1}\,\dd u 
\notag
\\
&=2^{1-a-b}\int_{-1}^{1} (1-t)^{a-1}\,(1+t)^{b-1}\,\dd t \, ,
\label{eq0.1}
\end{align}
and the following notation 
\begin{equation}
(a)_0 = 1,\, (a)_l = a(a+1)\dots (a+l-1)=\frac{\Gamma(\alpha +l)}{\Gamma(\alpha)}
\label{eq8.26*}
\end{equation}
for the falling factorial, see \cite[Sec.10.7, Eq.(11)]{2*}.
\begin{lemma}\label{lem4.1} 
For any space $Q=Q(d,d_0)$, 
the chordal metric $\tau$
has the following zonal spherical function expansion
\begin{equation}
\tau(y_1,y_2)=\frac12\sum\nolimits_{l\ge 1}\, M_l\,C_l\,\left[ 1-\phi_l(Q,x_1,x_2)\right],
\label{eq9.16**}
\end{equation}
where 
\begin{equation}
C_l = B(\alpha +3/2, \beta + l +1)\, 
\Gamma (l+1)^{-1}\, (1/2)_{l-1}\, P^{(\alpha,\beta)}_l (1) \, .
\label{eq0.16**}
\end{equation}
The series \eqref{eq9.16**} converges absolutely and uniformly.

\end{lemma}

\begin{proof}[Proof] 
Applying the formulas \eqref{eq0.17**} -- \eqref{eq0.19**} to the function $f(t)=(1-t)^{1/2}$ and using
\eqref{eq0.1}, we obtain the expansion 
\begin{align}
(1-t)^{1/2}& =2^{1/2}\,\Gamma(\alpha +3/2)\times
\notag
\\
&\sum\nolimits_{l\ge 0}\,
\frac{(2l+\alpha +\beta +1)\,\Gamma(l+\alpha +\beta +1)\,(-1/2)_l}
{\Gamma(l+\alpha +1)\,\Gamma(l+\alpha +\beta +5/2)}  \, 
\,P^{(\alpha,\beta)}_l(t).
\label{eq4.17}
\end{align}

Taking into account  \eqref{eq8.23} and \eqref{eq8.26*}, and applying Stirling's approximation to the gamma functions in \eqref{eq4.17},  
we observe
that the coefficients in \eqref{eq4.17} are of the order $O(l^{-2})$.
Therefore, the series \eqref{eq4.17} converges absolutely and uniformly.

Since  $(-1/2)_0=1$ and $(-1/2)_l=-1/2\,(1/2)_{l-1}$ for $l\ge 1$,  the series 
\eqref{eq4.17} can be written as follows
\begin{equation}
\left(\frac{1-t}{2}\right)^{1/2} = M_0\,C_0 - \frac12
\sum\nolimits_{l\ge 1}\,M_l\,C_l\,
\frac{P^{(\alpha,\beta)}_l (t)}{P^{(\alpha,\beta)}_l (1)}\, ,
\label{eq0.17}
\end{equation}
where $C_0=B(\alpha +3/2, \beta +1)$ and $C_l, l\ge 1$ are given in \eqref{eq0.16**}.
Putting $t=1$, we find
\begin{equation}
M_0\,C_0 = \frac12
\sum\nolimits_{l\ge 1}\,M_l\,C_l\, .
\label{eq0.18}
\end{equation}
Combining \eqref{eq0.17} and \eqref{eq0.18}, we obtain
\begin{equation}
\left(\frac{1-t}{2}\right)^{1/2} = \frac12 
\sum\nolimits_{l\ge 1}\,M_l\,C_l\,
\left[1-\frac{P^{(\alpha,\beta)}_l (t)}{P^{(\alpha,\beta)}_l (1)}\right]\, .
\label{eq0.19}
\end{equation}
For $t=\cos\theta (x_1, x_2)$, the equality \eqref{eq0.19} coincides with \eqref{eq9.16**}.
\end{proof}

\emph{Remark~4.1}. The expansion \eqref{eq4.17} can be found 
in \cite[Sec.10.20, Eq.(3)]{2*}.
However, it should be noted that  
the linear co-factor $(2l+\alpha +\beta +1)$ 
in \eqref{eq4.17} 
is misprinted in \cite[Sec.10.20, Eq.(3)]{2*}
as $\Gamma(2l+\alpha +\beta +1)$.

\begin{lemma}\label{lem4.2} 
For any space $Q=Q(d,d_0)$ and any finite measure $\xi$ on $[0,\pi]$, 
the symmetric difference metric \eqref{eq1.13}
has the following zonal spherical function expansion

\begin{equation}
\theta^{\Delta}(\xi, y_1,y_2) =
B(d/2,d_0/2)^{-1} \sum\nolimits_{l\ge 1}\, l^{-2} M_lA_l(\xi) 
\left[ 1-\phi_l(Q,x_1,x_2)\right],
\label{eq9.16}
\end{equation}
where 
\begin{equation}
A_l(\xi)=\int^{\pi}_{0}  
(\sin\frac 12 r)^{2d}(\cos \frac12 r)^{2d_0}
\left\{ P^{(\alpha+1,\beta+1)}_{l-1} (\cos r)\right\}^2 
\, \dd \xi (r). 
\label{eq9.16*}
\end{equation}
The series \eqref{eq9.16} converges absolutely and uniformly.
\end{lemma}

The expansion \eqref{eq9.16} has been established in \cite[Theorem~4.1(ii)]{31}.
The proof is based on the observation 
that the term $\mu (\BBB(y_1,r)\cap \BBB(y_2,r))$ in the formula \eqref{eq1.14*} can be thought of as a convolution of the characteristic functions of the balls on the homogeneous space $Q(d,d_0)$. 

\begin{proof}[Proof of Theorem~1.2.] 
Substituting \eqref{eq9.16**} and \eqref{eq9.16} into \eqref{eq2.8} and 
equating coefficients at each $\phi_l$, we obtain the following series of equations  
\begin{align}
\gamma (Q) A_l(\xi^{\natural}) = \frac{(1/2)_{l-1} l^2}{2\Gamma(l+1)}\,
\begin{pmatrix} \alpha +l \\ l  \end{pmatrix}\,
B(d/2, d_0/2)\, B((d+1)/2, l+d_0/2)\, ,
\label{eq0.8}
\end{align}
for all $l\ge 1$. If the measure $\dd\xi^{\natural} (r)=\sin (r)\dd r$,
the integral \eqref{eq9.16*} takes the form
\begin{equation}
A_l(\xi^{\natural})=2\int^{\pi}_{0}  
(\sin\frac 12 r)^{2d+1}(\cos \frac12 r)^{2d_0+1}
\left\{ P^{(\alpha+1,\beta+1)}_{l-1} (\cos r)\right\}^2 
\, \dd r \, . 
\label{eq0.20}
\end{equation}
Each of the equations \eqref{eq0.8} can be used to determine the constant $\gamma(Q)$.
In the simplest case of $l=1$, we have
\begin{equation}
\gamma (Q) A_1(\xi^{\natural}) = \frac{d}{4}\, B(d/2, d_0/2)\,
B((d+1)/2, 1+d_0/2)\, ,
\label{eq0.21}
\end{equation}
where
\begin{equation}
A_1(\xi^{\natural})=2\int^{\pi}_{0}  
(\sin\frac 12 r)^{2d+1}(\cos \frac12 r)^{2d_0+1}\, \dd r = 2 B(d+1, d_0 +1)\, , 
\label{eq0.22}
\end{equation}
see \eqref{eq0.1}. Therefore,
\begin{align}
\gamma (Q) = \frac{d\,B(d/2, d_0/2)\,B((d+1)/2, 1+d_0/2)}{8B(d+1, d_0 +1)}\, .
\label{eq0.22*}
\end{align}
In the terms of gamma functions, we have
\begin{align}
\gamma (Q)=\frac{\Gamma(d/2)\,\Gamma(d_0/2)^2\,\Gamma((d+1)/2)\,\Gamma(d+d_0 +2)}
{16\,\Gamma(d)\,\Gamma(d_0)\,\Gamma((d+d_0)/2)\,\Gamma((d+d_0 +3)/2)} \, , 
\label{eq0.22**}
\end{align} 
where the relation $\Gamma(z+1)=z\Gamma(z)$ has been used. Applying 
the duplication formula
$
\Gamma (2z)=\pi^{-1/2}2^{2z-1} \Gamma (z)\Gamma(z+1/2) 
\label{eq9.8}
$
to the terms $\Gamma(d)\, ,\Gamma(d_0)$ and $\Gamma(d+d_0 +2)$,
we obtain
\begin{align}
\gamma(Q)&=\frac{\sqrt{\pi}}{4}\,(d+d_0)\,\frac{\Gamma(d_0/2)}{\Gamma((d_0+1)/2)} \, .
\label{eq4.33**}
\end{align}
This completes the proof.
\end{proof}

\emph{Remark~4.2}. The equalities \eqref{eq0.8} with the constant 
\eqref{eq4.33**} define explicit formulas for the integrals \eqref{eq0.20}. We have
\begin{align}
\int^1_{-1}
\left(P^{(d/2, d_0/2)}_{l-1}(t)\right)^2 \,& \left(\frac{1-t}{2}\right)^{d}
\left(\frac{1+t}{2}\right)^{d_0}\, \dd t 
= \frac{2\, (1/2)_{l-1}}{(l-1)!}\,B(d+1, d_0 +1)\times
\notag
\\
\notag
\\
&\times \frac{\Gamma(d/2+l)\,\Gamma(d_0/2+l)\,\Gamma(d/2+d_0/2+3/2))}
{\Gamma(d/2+1)\,\Gamma(d_0/2+1)\,\Gamma(d/2+d_0/2+1+l)} 
\label{eq0.81}
\end{align}
for all $l\ge 1$.

It is worth noting that a direct proof of the formulas
\eqref{eq0.81} makes possible an alternative proof of Theorems~1.1 and 1.2.
Indeed, the formulas \eqref{eq0.81} together with
the expansions \eqref{eq9.16**} and \eqref{eq9.16}  imply the relations
\eqref{eq2.8} and \eqref{eq1.33*}. Such an approach to the proof of
invariance principle will be realized in the second part of this work.


\end{document}